\documentclass[oneside,english]{amsart}
\usepackage[T1]{fontenc}
\usepackage[latin9]{inputenc}
\usepackage{verbatim}
\usepackage{amstext}
\usepackage{amsthm}
\usepackage{amssymb}

\makeatletter
\numberwithin{equation}{section}
\numberwithin{figure}{section}
  \theoremstyle{plain}
  \newtheorem*{thm*}{\protect\theoremname}
 \theoremstyle{definition}
 \newtheorem*{defn*}{\protect\definitionname}
\theoremstyle{plain}
\newtheorem{thm}{\protect\theoremname}
  \theoremstyle{plain}
  \newtheorem{lem}[thm]{\protect\lemmaname}
  \theoremstyle{definition}
  \newtheorem{defn}[thm]{\protect\definitionname}
  \theoremstyle{plain}
  \newtheorem{prop}[thm]{\protect\propositionname}
  \theoremstyle{plain}
  \newtheorem{cor}[thm]{\protect\corollaryname}

\makeatother

\usepackage{babel}
  \providecommand{\corollaryname}{Corollary}
  \providecommand{\definitionname}{Definition}
  \providecommand{\lemmaname}{Lemma}
  \providecommand{\propositionname}{Proposition}
  \providecommand{\theoremname}{Theorem}
\providecommand{\theoremname}{Theorem}

\begin{document}

\title{Normal Subgroups Of Powerful $p$-Groups}

\author{James Williams}
\begin{abstract}
In this note we show that if $p$ is an odd prime and $G$ is a powerful
$p$-group with $N\leq G^{p}$ and $N$ normal in $G$, then $N$
is powerfully nilpotent. An analogous result is proved for $p=2$
when $N\leq G^{4}$.
\end{abstract}

\maketitle

\section{Introduction}

Powerfully nilpotent groups, introduced in \cite{Traustason2018},
are a type of powerful $p$-group, possessing a central series of
a special kind. This family of groups have a rich and beautiful structure
theory. For example to each powerfully nilpotent group we can associate
a quantity known as the \emph{powerful coclass}, and it turns out
that the rank and exponent of a powerfully nilpotent group can be
bounded in terms of their powerful coclass. Many characteristic subgroups
of powerful $p$-groups are in fact powerfully nilpotent. In \textbf{\cite{Traustason2018}}
it is shown that if $G$ is powerful then the proper terms of the
derived and lower central series of $G$ are powerfully nilpotent,
as well as $G^{p^{i}}$ for $i\geq1$. In \textbf{\cite{Williams2018}}
the Omega subgroups of $G^{p^{i}}$ are shown to be powerfully nilpotent.
In this note we prove the following for odd primes $p$:
\begin{thm*}
For an odd prime $p$, if $G$ is a powerful $p$-group and $N\leq G^{p}$
with $N$ normal in $G$, then $N$ is powerfully nilpotent. 
\end{thm*}
For $p=2$ we prove:
\begin{thm*}
If $G$ is a powerful $2$-group and $N\leq G^{4}$ with $N$ normal
in $G$, then $N$ is powerfully nilpotent.
\end{thm*}
This builds on the results in\textbf{ }\cite{GONZALEZSANCHEZ2004193}
, which show that such a group $N$ must be powerful and moreover
provides an alternative proof for this fact. We also note that in
\cite{article}, Mann poses the question: \emph{``Which $p$-groups
are subgroups of powerful $p$-groups?''}, and our result provides
a partial answer in this direction.

\section{Preliminaries}

In what follows all groups considered will be finite $p$-groups.
First we shall recall some definitions and properties which will be
used in the main part of this paper, with the aim of making this paper
as self contained as possible. Recall that $H^{n}=\langle x^{n}|x\in H\rangle$
is the subgroup generated by $n$th powers of elements of $H$. 
\begin{defn*}
A subgroup $N$ of a finite $p$-group $G$ is \emph{powerfully embedded}
in $G$ if $[N,G]\leq N^{p}$ for $p$ an odd prime, or $[N,G]\leq N^{4}$
if $p=2$.
\end{defn*}

\begin{defn*}
A finite $p$-group $G$ is \emph{powerful} if $[G,G]\leq G^{p}$
for $p$ an odd prime, or $[G,G]\leq G^{4}$ if $p=2$.
\end{defn*}
Powerful $p$-groups were introduced in \cite{Lubotzky1987}. We will
often make use of the following well known properties of powerful
$p$-groups without explicit mention.

\begin{thm}
Let $G$ be a powerful $p$-group, then
\begin{enumerate}
\item \cite[Proposition 1.7]{Lubotzky1987} for every $k\in\mathbb{N}$
the subgroup $G^{p^{k}}$ coincides with the set $\{x^{p^{k}}|x\in G\}$
of $p^{k}th$ powers of elements of $G$.
\item \cite[Corollary 1.2]{Lubotzky1987} $G^{p^{k}}$ is powerfully embedded
in $G$ for all $k\in\mathbb{N}.$
\item \cite[Corollary 1.9]{Lubotzky1987} Suppose that $G=\langle a_{1},\dots,a_{r}\rangle$
is generated by elements $a_{1},\dots,a_{r}$. Then $G^{p}=\langle a_{1}^{p},\dots,a_{r}^{p}\rangle$.
\end{enumerate}
\end{thm}

We also make use of the following result
\begin{lem}
\cite[Lemma 3.1]{SHALEV1993271} If $M$ and $N$ are powerfully embedded
subgroups in a finite $p$-group $G$, then $[M^{p^{i}},N^{p^{j}}]=[M,N]^{p^{i+j}}$
for all $i,j\in\mathbb{N}$.
\end{lem}

We keep with the convention of \cite{Fernandez-Alcober2007} that
if $G$ is a $p$-group, and $x\in G$, we define the meaning of the
inequality $o(x)\leq p^{i}$ with $i<0$ to be that $x=1$. In \cite{Fernandez-Alcober2007},
the following result is proved. This result forms an essential part
of our argument, and holds for all primes $p$.
\begin{thm}[Fern\'andez-Alcober]
\label{thm:GustavosTheorem}Let $G$ be a powerful $p$-group. Then,
for every $i\geq0$ if $x,y\in G$ are such that $o(x)\leq p^{i+1}$
and $o(y)\leq p^{i},$ then $o([x^{p^{j}},y^{p^{k}}])\leq p^{i-j-k}$
for all $j,k\geq0$.
\end{thm}

The following result is proved in \cite{GONZALEZSANCHEZ2004193}.
\begin{thm}
\label{thm:normal subgroups of g contained in frattini are powerful}Let
$G$ be a powerful $p$-group.
\begin{enumerate}
\item Let $p$ be an odd prime. If $N\vartriangleleft G$ and $N\leq G^{p}$
then $N$ is powerful. 
\item Let $p=2$. If $N\vartriangleleft G$ and $N\leq G^{4}$ then $N$
is powerful.
\end{enumerate}
\end{thm}

For further details on powerful $p$-groups, see \cite{Dixon2003,Khukhro1998,Lubotzky1987}. 

In \textbf{\cite{Traustason2018}} the notion of a \emph{powerfully
nilpotent group }was introduced. 
\begin{defn}
Let $H\leq K\leq G$. An ascending chain of subgroups 
\[
H=H_{0}\leq H_{1}\leq\dots\leq H_{n}=K
\]
is \emph{powerfully central} if $[H_{i},G]\leq H_{i-1}^{p}$ for $i=1,\dots,n$.
A powerful $p$-group $G$ is \emph{powerfully nilpotent} if it has
a powerfully central ascending chain of subgroups of the form
\[
\{1\}=H_{0}\leq H_{1}\leq\dots\leq H_{n}=G.
\]

\end{defn}

A fastest ascending powerfully central series, named the \emph{upper
powerfully central series} is defined recursively as follows: $\hat{Z}_{0}(G)=\{1\}$
and for $n\geq1$ 
\[
\hat{Z}_{n}(G)=\{a\in G:[a,x]\in\hat{Z}_{n-1}(G)^{p}\text{ for all \ensuremath{x\in G}}\}.
\]
 Notice in particular that $\hat{Z}_{1}(G)=Z(G)$.

The following result, proved in \textbf{\cite[Section 1, Proposition 1.1]{Traustason2018}},
allows us to reduce the problem of whether or not a group is powerfully
nilpotent to looking at a group of exponent $p^{2}$. 
\begin{prop}
\label{prop:G pn iff G/Gp2 is pn}Let $G$ be any finite $p$-group,
then $G$ is powerfully nilpotent if and only if $G/G^{p^{2}}$ is
powerfully nilpotent. 
\end{prop}
The following reduction is also needed. 
\begin{prop}
\label{prop:G pn iff G/Z(G)^p is pn}$G$ is powerfully nilpotent
if and only if $G/Z(G)^{p}$ is powerfully nilpotent.\end{prop}
\begin{proof}
This follows immediately from the definition of the upper powerfully
central series.
\end{proof}

\section{Main Result}

We begin this section with a remark on the prime $p=2$. It was shown
in \textbf{\cite{Traustason2018}}, that any powerful $2$-group is
powerfully nilpotent. By Theorem \ref{thm:normal subgroups of g contained in frattini are powerful}
any normal subgroup of $G$ contained within $G^{4}$ is powerful
and thus will be powerfully nilpotent. However we give an alternate
proof of the fact that this group is powerful, for $p=2$ below, using
ideas in line with the rest of this paper.
\begin{thm}
Let $G$ be a powerful $2$-group and $N\leq G^{4}$ with $N$ normal
in $G$. Then $N$ is powerfully nilpotent.\end{thm}
\begin{proof}
It suffices to show that $\frac{N}{N^{4}}$ is abelian. Let $\bar{N}=\frac{N}{N^{4}}$
and $\bar{G}=\frac{G}{N^{4}}$. Notice that $\bar{G}$ is powerful.
Consider two elements $a,b\in\bar{N}$ and observe that the order
of both elements is at most $4$. Notice also that $\bar{N}\leq\bar{G}^{4}$,
so we may assume $a=g^{2^{2}}$ and $b=h^{2^{2}}$ for some $g,h\in\bar{G}$
as the group $\bar{G}$ is powerful. Thus we must have $o(g)\leq2^{4}$
and $o(h)\leq2^{4}$. Then $[a,b]=[g^{2^{2}},h^{2^{2}}]=1$ by Theorem
\ref{thm:GustavosTheorem}, setting $x=g$, $y=h$ and $i=4$. It
follows that $\bar{N}$ is abelian and so $N$ is a powerful $2$-group.
Then by the remark in \textbf{\cite[page 81]{Traustason2018}}, we
know that powerful $2$-groups are powerfully nilpotent and thus\textbf{
}we have that $N$ is powerfully nilpotent.
\end{proof}
In what follows let $p$ be an odd prime.
\begin{lem}
\label{lem:In a group of exponent p^2 elements of order p in N are central}If
$G$ is powerful, $N\leq G^{p}$ and $N$ has exponent at most $p^{2}$
then elements of order $p$ in $N$ are central in $N$. 
\end{lem}
An equivalent result was proved in \cite{Williams2018}, but the short
proof is included here for completeness.
\begin{proof}
This follows from Theorem \ref{thm:GustavosTheorem} above. Indeed
if $g^{p},h^{p}\in N$ and $o(g)\leq p^{2}$ and $o(h)\leq p^{3}$
then in Theorem \ref{thm:GustavosTheorem}, take $i=2$, $x=h$ and
$y=g$.\end{proof}
\begin{lem}
\label{lem:intersection gp2 and N is central in N}Let $G$ be a powerful
$p$-group and $N\leq G^{p}$ and suppose $N$ has exponent $p^{2}$.
Then any element in $G^{p^{2}}\cap N$ is central in $N$. \end{lem}
\begin{proof}
By Lemma \ref{lem:In a group of exponent p^2 elements of order p in N are central}
above, we know we only need to consider commutators between elements
of order $p^{2}.$ Suppose $n_{1}\in N\cap G^{p^{2}}$ and $n_{2}\in N\leq G^{p}$
are both of order $p^{2}$. Then we may write $n_{1}=g^{p^{2}}$ and
$n_{2}=h^{p}$ for some $g,h\in G$, where $o(g)=p^{4}$ and $o(h)=p^{3}$.
Then by Theorem \ref{thm:GustavosTheorem} above, we can set $x=g$,
and $y=h$ and $i=3$ and we see that $[x^{p^{2}},y^{p}]=1$.
\end{proof}

We now prove the following crucial lemma.
\begin{lem}
\label{lem:N must contain central element of order p^2}Let $G$ be
a powerful $p$-group and let $N\leq G^{p}$ where $N$ is normal
in $G$ and the exponent of $N$ is $p^{2}$. Then $N$ contains an
element of order $p^{2}$ which is central in $N$. \end{lem}
\begin{proof}
First observe that by Lemma \ref{lem:In a group of exponent p^2 elements of order p in N are central},
we can assume that $N$ contains at least two elements of order $p^{2}$
which do not commute with each other, or else $N$ would be abelian
and the result is trivial. Suppose these elements are $n_{1}$ and
$n_{2}$. By Lemma \ref{lem:intersection gp2 and N is central in N}
we can assume that these are in $G^{p}\backslash G^{p^{2}}$. They
are $p$th powers, since $G$ is powerful, so we may write $n_{1}=a^{p}$
and $n_{2}=b^{p}$ for some $a,b\in G$ where $o(a)=p^{3}=o(b)$.
Now notice that $[b,a^{p}]\in N$ since $N$ is normal in $G$ and
also observe that since $G$ is powerful, we have that $[b,a^{p}]\in[G,G^{p}]\leq G^{p^{2}}$.
Hence $[b,a^{p}]\in N\cap G^{p^{2}}$. We seek to show that this element
has order $p^{2}$. Then the result will follow by Lemma \ref{lem:intersection gp2 and N is central in N}.
Thus we argue by contradiction supposing that $[b,a^{p}]$ is of order
$p$.

We will now show that $[b,a^{p}]^{p}=[b^{p},a^{p}]$. We first show
that the subgroup $M=\langle b,[b,a^{p}]\rangle$ is of class at most
$2$. Notice that $[b,a^{p}]\in G^{p^{2}}$ and so we let $[b,a^{p}]=g^{p^{2}}$.
Furthermore, since $a^{p}\in N$ and $N$ is normal in $G$ we have
that $[b,a^{p}]\in N$ and since the exponent of $N$ is $p^{2}$
we can conclude that the order of $g$ is at most $p^{4}$. Using
Theorem \ref{thm:GustavosTheorem} we see that $o([g^{p^{2}},b])\leq p^{3-2}=p$.
Thus since $[g^{p^{2}},b]\in G^{p^{3}}$we can let $[g^{p^{2}},b]=h^{p^{3}}$
where $o(h)\leq p^{4}$. 

Next notice that $[b,a^{p}]\in N$ and $h^{p^{3}}=[[b,a^{p}],b]\in N$
are both of order at most $p$, then by Lemma \ref{lem:In a group of exponent p^2 elements of order p in N are central}
they commute with each other. Finally we consider $[h^{p^{3}},b]=[[[b,a^{p}],b],b]$.
As $o(h)\leq p^{4}$ and $o(b)\leq p^{3}$, again appealing to Theorem
\ref{thm:GustavosTheorem} we see that this commutator is trivial.
Hence we have that $[[[b,a^{p}],b],[b,a^{p}]]$ and $[[[b,a^{p}],b],b]$
are both trivial. We can then conclude that $M$ has class at most
$2$ and $M^{\prime}=\langle[[b,a^{p}],b]\rangle$ is cyclic of order
at most $p$. 

Since $M$ has class $2$ then $[b^{p},a^{p}]=[b,a^{p}]^{p}\cdot[b,a^{p},b]^{\binom{p}{2}}$
and as $p$ is odd we have $p\mid\binom{p}{2}$ and thus it follows
that $[b,a^{p}]^{p}=[b^{p},a^{p}]$. Now as we assumed that $a^{p}$
and $b^{p}$ did not commute, this means that $o([b^{p},a^{p}])\geq p$.
Hence we see that we must have $o([b,a^{p}])=p^{2}$, but then $[b,a^{p}]=g^{p^{2}}$
is an element of order $p^{2}$ in $N\cap G^{p^{2}}$ and so must
be central in $N$ by Lemma \ref{lem:intersection gp2 and N is central in N}. 
\end{proof}
We can now prove the main result.
\begin{thm}
Let $G$ be a powerful $p$-group with $p$ an odd prime. Let $N\leq G^{p}$
and $N\vartriangleleft G$. Then $N$ is powerfully nilpotent. \end{thm}
\begin{proof}
The proof is by induction on the order of $G$. The result is clearly
true for $|G|\leq p^{2}$. Suppose now that $|G|=p^{k}$ with $k>2$
and that the claim holds for all smaller orders. Consider a normal
subgroup $N$ in $G$ such that $N\leq G^{p}$. We will show that
$N$ is powerfully nilpotent. Notice that $N^{p^{2}}$ is characteristic
in $N$ and so is normal in $G$. Therefore $\frac{N}{N^{p^{2}}}\leq\left(\frac{G}{N^{p^{2}}}\right)^{p}$
and $\frac{N}{N^{p^{2}}}$ is normal in $\frac{G}{N^{p^{2}}}$ . If
$|N^{p^{2}}|>1$ then by the inductive hypothesis and Proposition
\ref{prop:G pn iff G/Gp2 is pn} it would follow that $N$ is powerfully
nilpotent. Thus we may assume that $|N^{p^{2}}|=1$. We remark also
that we may assume the exponent of $N$ is precisely $p^{2}$, for
otherwise $N$ would be abelian by Lemma \ref{lem:In a group of exponent p^2 elements of order p in N are central}. 

Next notice that $Z(N)^{p}$ is characteristic in $N$ and so is normal
in $G$ and is contained in $G^{p}$. By Lemma \ref{lem:N must contain central element of order p^2}
we know that $|Z(N)^{p}|\geq p$. Then observe that $\frac{N}{Z(N)^{p}}\leq\left(\frac{G}{Z(N)^{p}}\right)^{p}$
and $\frac{N}{Z(N)^{p}}\vartriangleleft\frac{G}{Z(N)^{p}}$ and as
$\frac{G}{Z(N)^{p}}$ is a group of strictly smaller order, it follows
by the induction hypothesis that $\frac{N}{Z(N)^{p}}$ is powerfully
nilpotent. Then by Proposition \ref{prop:G pn iff G/Z(G)^p is pn}
it follows that $N$ is powerfully nilpotent.
\end{proof}
The next corollary is immediate and shows that powerfully nilpotent
groups are abundant within powerful $p$-groups.
\begin{cor}
Let $G$ be a powerful $p$-group and $i\in\mathbb{N}$. If $N$ is
normal in $G^{p^{i}}$ and contained within $G^{p^{i+1}}$, then $N$
is powerfully nilpotent.
\end{cor}
This corollary holds for $p=2$ if $G^{p^{i+1}}$ is replaced with
$G^{p^{i+2}}$.

\section{Acknowledgements}

The author is very grateful to Gustavo Fern\'andez-Alcober for conjecturing
the result proved in this paper and for his comments on this manuscript.
The author would like to thank Gunnar Traustason for his consistently
excellent supervision throughout his PhD and for noting an error in
a preliminary draft of this manuscript. The author also wishes to
thank Gareth Tracey for his feedback and critical reading of this
paper.

\bibliographystyle{amsplain}
\bibliography{C:/Users/flowe/OneDrive/PhD/WritingUp/Thesis/referenceDatabase}

\end{document}